\documentclass[11pt, a4paper,leqno]{amsart}
\usepackage{amsmath,amsthm,amscd,amssymb,amsfonts, amsbsy}
\usepackage{latexsym}
\usepackage{txfonts}
\usepackage{exscale}
\usepackage{mathrsfs}
\usepackage{graphicx}
\usepackage{huncial}
\usepackage[T1]{fontenc}
\usepackage[titletoc]{appendix}

\usepackage[colorlinks,citecolor=red,pagebackref,hypertexnames=false, breaklinks]{hyperref}
\usepackage[usenames,dvipsnames]{color}

\newcommand{\abs}[1]{{\left\lvert{#1}\right\rvert}}
\newcommand{\norm}[1]{{\left\lVert{#1}\right\rVert}}
\newcommand{\br}[1]{{\left({#1}\right)}}

\calclayout
\allowdisplaybreaks

    \def\XXint#1#2#3{{\setbox0=\hbox{$#1{#2#3}{\int}$}
    \vcenter{\hbox{$#2#3$}}\kern-.5\wd0}}

\def \E{ \mathbb{E} }

\def \R{ \mathbb{R} }
\def \N{ \mathbb{N} }

\def \M { \mathcal{M} }

\def \Int{\,\Int\,}

\renewcommand{\Int}{{\rm Int}\,}

\renewcommand{\chi}{{\bf 1}}

\theoremstyle{plain}
\newtheorem{theorem}{Theorem}
\newtheorem{lemma}[theorem]{Lemma}
\newtheorem{corollary}[theorem]{Corollary}

\theoremstyle{definition}

\theoremstyle{remark}
\newtheorem{remark}[theorem]{Remark}

\numberwithin{equation}{section}
\numberwithin{theorem}{section}

\author{Li Chen}
\address{Li Chen, Instituto de Ciencias Matem\'aticas CSIC-UAM-UC3M-UCM, Consejo Superior de Investigaciones Cient{\'\i}ficas,
C/ Nicol\'as Cabrera, 13-15, E-28049 Madrid, Spain} 
\email{li.chen@icmat.es}

\author{Thierry Coulhon}
\address{Thierry Coulhon,  PSL  Research University,  F-75005 Paris, France}
\email{thierry.coulhon@univ-psl.fr}

\author{Joseph Feneuil}
\address{Joseph Feneuil, Department of mathematics, University of Minnesota, USA}
\email{jfeneuil@umn.edu}

\author{Emmanuel Russ}
\address{Emmanuel Russ, Institut Fourier, UMR 5582, Grenoble Alpes University, France}
\email{emmanuel.russ@ujf-grenoble.fr}

\begin{document}
\title{Riesz transform for $1\leq p \le  2$ without Gaussian heat kernel bound
{\small }}


\thanks{The first author was partially supported by ICMAT Severo Ochoa project SEV-2011-0087 and she acknowledges that the research leading to these results has received funding from the European Research Council under the European Union's Seventh Framework Programme (FP7/2007-2013)/ ERC agreement no. 615112 HAPDEGMT. The third and fourth authors were partially supported by the French ANR project ``HAB'' no.~ANR-12-BS01-0013.}

\date{\today}

\begin{abstract}
We study the $L^p$ boundedness of Riesz transform as well as the reverse inequality on Riemannian  manifolds and graphs under the volume doubling property and 
a sub-Gaussian heat kernel upper bound. 
We prove that  the Riesz transform is  then  bounded on $L^p$  for $1<p<2$, which  shows that Gaussian estimates of the heat kernel are not a necessary condition for this.
In the particular case of Vicsek manifolds and graphs, we show that the reverse inequality does not hold for $1<p<2$. This yields a full picture of the ranges of $p\in (1,+\infty)$ for which respectively the  Riesz transform is $L^p$ -bounded and  the reverse inequality holds on $L^p$  on such manifolds and graphs. This picture is strikingly different from the Euclidean one.  
 
\end{abstract}

\maketitle

\tableofcontents

\section{Introduction}

Let $M$ be a complete connected non-compact Riemannian manifold. Let $d$ be the geodesic distance and $\mu$ be the Riemannian measure.  Denote by $B(x,r)$ the  open  ball of center $x$ and of geodesic radius $r$. We write $V(x,r)$ for $\mu(B(x,r))$.
Let $\nabla$ be the Riemannian gradient and $\Delta$ be the non-negative Laplace-Beltrami operator on $M$. 
Denote by $(e^{-t\Delta})_{t>0}$ the heat semigroup associated with $\Delta$ and $h_t(x,y)$ the associated heat kernel. It is well-known that
$h_t(x,y)$ is everywhere positive, symmetric in $x,y\in M$ and smooth in $t>0, x,y\in M$ (see for instance \cite[Theorem 7.13]{Gr09}).
For $f \in \mathcal C_0^{\infty}(M)$,  denote by  $|\nabla f|$ the Riemannian length of  $\nabla f$. From the definition and  spectral theory, it holds that for all $f \in \mathcal C_0^{\infty}(M)$,
\begin{equation}\label{2}
\||\nabla f|\|_2^2 =(\Delta f,f)=\norm{\Delta^{1/2} f}_2^2.
\end{equation}

In the case where $M=\R^n$ endowed with the Euclidean metric, it is well-known that this equality extends to $L^p$, $1<p<+\infty$, in the form of an equivalence of seminorms
\begin{equation} \tag{E$_p$} \label{Ep} \norm{|\nabla f|}_p\simeq\norm{\Delta^{1/2} f}_p, \end{equation}
for $f\in\mathcal{C}_0^{\infty}(\R^n)$. Here and in the sequel we use the notation  $u\simeq v$ if $u\lesssim v$ and $v\lesssim u$, where $u\lesssim v$ means that there exists a constant $C$ (independent of the important parameters) such that $u\leq Cv$.\

It was asked by Strichartz \cite{Str83} in 1983 on which non-compact Riemannian  manifolds  $M$, and for which $p$, $1<p<+\infty$, 
the equivalence of seminorms \eqref{Ep} still holds on $\mathcal{C}_0^{\infty}(M)$. 
Let us distinguish  between the two inequalities involved in \eqref{Ep} and introduce a specific notation.

The first one 
\begin{equation} \label{Rp} \tag{$R_p$}
\norm{|\nabla f|}_p \lesssim\norm{\Delta^{1/2} f}_p, \,\,\forall f \in \mathcal C_0^{\infty}(M),
\end{equation}
can be reformulated by saying  that the Riesz transform $\nabla\Delta^{-1/2}$ is $L^p$ bounded on $M$.
We will refer to the second one
\begin{equation} \label{RRp} \tag{$RR_p$}
\norm{\Delta^{1/2} f}_p \lesssim \norm{|\nabla f|}_p, \,\,\forall f \in \mathcal C_0^{\infty}(M),
\end{equation}
as the reverse inequality.

It is well-known (see for example \cite[Proposition 2.1]{CD03}) that by duality (\ref{Rp}) implies $(RR_{p'})$, where $p'$ is the conjugate exponent  of $p$.

\medskip
We say that $M$ satisfies the  volume  doubling property if, for any $x\in M$ and $r>0$,
\begin{equation}\label{doubling}
V(x,2r)\lesssim V(x,r).\tag{$D$}
\end{equation}

 The  key result on (\ref{Rp}) in the case  $1<p<2$ is the following one.

\begin{theorem}\cite[Theorem 1.1]{CD99} \label{thm-cd}
Let $M$ be a complete non-compact Riemannian manifold satisfying the volume  doubling property (\ref{doubling}) and the heat kernel  upper estimate 
\begin{equation} \label{UE} \tag{$U\!E$}
h_t(x,y)\lesssim \frac{1}{V(x,\sqrt t)} \exp\left(-c\frac{d^2(x,y)}{t}\right),\,\,\forall x,y\in M,\, t>0,
\end{equation}
for some $c>0$.
Then (\ref{Rp}) holds for $1<p<2 $.
\end{theorem}

Note that one does not know any  complete Riemannian manifold where  (\ref{Rp}) does not hold for some  $p\in(1,2)$. It is therefore very natural to ask whether the above assumptions are necessary or not.
It is clear that \eqref{doubling} is not. For instance, as soon as the manifold has a spectral gap (and satisfies local volume doubling and small time gaussian estimates), it is easy to see that  \eqref{Rp} holds (see \cite[Theorem 1.3]{CD99}); the presence of a spectral gap is clearly incompatible with global volume doubling. Moreover it is know that \eqref{Rp} holds for all $p\in (1,2)$ on some specific Lie groups of exponential growth (see for instance  \cite{Sjo99,SV08}).
Yet, if \eqref{doubling} is assumed, are Gaussian estimates of the heat kernel necessary for the Riesz transform to be $L^p$-bounded for $1< p< 2$?  As  a matter of fact, no examples  were  known so far  where (\ref{Rp}) holds for $1<p<2$ and \eqref{doubling} holds without \eqref{UE} holding as well.

The main goal of the present paper is to give a negative answer  to the above question.
We will deduce this from a positive result: 
we  will in fact establish that
 (\ref{Rp})  also holds for all $p\in (1,2)$ if we replace (\ref{UE}) in the assumptions of Theorem \ref{thm-cd} by a
 so-called sub-Gaussian heat kernel  upper estimate. 
 The conclusion follows because  there exist  manifolds that satisfy  such sub-Gaussian upper estimates as well as the matching lower estimates and the latter are clearly incompatible with \eqref{UE} (see the Appendix below for both facts).
 As a by-product, this provides new classes of examples where  (\ref{Rp}) does hold for  $p\in (1,2)$.
 
 Note that the question remains whether any heat kernel estimate at all is necessary. A bold conjecture has been made in this direction  in   \cite[Conjecture 1.1]{CD03}. 
 
 Let $m>2$.  One says that $M$ satisfies the sub-Gaussian heat kernel  upper  estimate with  exponent $m$  if for any $x,y \in M$,
\begin{equation}\label{B}\tag{$U\!E_m$}
 h_{t}(x,y) \lesssim\left\{ \begin{aligned}
& \frac{1}{V(x,\sqrt t)}\exp\br{-c\frac{d^{2}(x,y)}{t}}, &0<t<1,\\
& \frac{1}{V(x,t^{1/m})}\exp\br{-c \br{\frac{d^{m}(x,y)}{t}}^{1/(m-1)}}, &t\geq 1,
\end{aligned}\right.
\end{equation}
for some $c>0$.

We refer to \cite[Section 5]{HSC01} for a detailed introduction  of a class of non-classical heat kernel estimates that includes this one.

Fractal manifolds are typical examples that satisfy sub-Gaussian heat kernel estimates; they are built from graphs with a self-similar structure at infinity by replacing the edges of the graph with tubes of length $1$ and then gluing the tubes together smoothly
at the vertices (see \cite{BCG01}). It has been proved in \cite{Ba04} that for any $D,m\in \R$ such that $D> 1$ and  $2< m\leq D+1$, there exists an infinite
connected locally finite graph with polynomial growth of exponent $D$ satisfying  the discrete analogue of (\ref{B}) (see Section \ref{Rieszgraph} for a precise definition).  It follows that for any $D,m\in \R$ such that $D> 1$ and  $2< m\leq D+1$, there exist complete connected Riemannian manifolds satisfying $V(x,r) \simeq r^D$ for $r\geq 1$ and \eqref{B} (see the Appendix below).

  The article  \cite{F15}  by the third author is the main source of inspiration for the present paper. It treats the  analogous  problem in a discrete setting:   \cite[Corollary 1.30]{F15}  states  that the  discrete  Riesz transform is $L^p$ bounded for $1<p<2$ on graphs satisfying the  volume doubling property and  a  sub-Gaussian  upper estimate for the Markov kernel. The author introduces a Hardy space theory for functions and $1-$forms on graphs  and obtains the $H^1$-boundedness of the Riesz transform. Then its  $L^p$-boundedness for $1<p<2$  follows  by interpolation  between $H^1$ and $L^2$.   We rely  on a crucial idea from \cite{F15}, namely a new way to use  a well-known and very powerful trick by Stein (\cite[Chapter 2, Lemma 2]{St70}; for other utilizations of the argument in our subject see \cite[Theorem 1.2]{CDL}   and \cite[Proposition 1.8]{ChThese}). Instead of introducing $H^1$ spaces as in \cite{F15}, we follow a more direct route  and prove as in \cite{CD99}  that the Riesz transform is  weak $(1,1)$ by adapting the method to the sub-Gaussian case as in \cite{Ch15}. This proof also works in the discrete case and gives a simpler proof of the  case $1<p<2$  in  \cite[Corollary 1.30]{F15}.

Our main result is:
\begin{theorem}\label{main}
Let $M$ be a complete non-compact Riemannian manifold satisfying the doubling property (\ref{doubling}) and the sub-Gaussian heat kernel upper 
bound (\ref{B}) for some $m>2$. Then the Riesz transform is weak $(1,1)$ and $L^p$ bounded for $1<p\leq 2$.
\end{theorem}

\noindent  Recall that a (sub)-linear operator defined on $L^1$ is said to be weak $(1,1)$ if, for all $\lambda>0$,
$$
\mu\left(\left\{x\in M;\ \left\vert Tf(x)\right\vert>\lambda\right\}\right)\lesssim \frac 1{\lambda} \left\Vert f\right\Vert_1.
$$
 The plan of the paper is as follows. Section \ref{integratestim} is devoted to  a crucial integral estimate for the gradient of the heat kernel. In Section \ref{weak11}, we prove our main result by using the Duong-McIntosh singular integral method from \cite{DM99} as in \cite{CD99}.  As a consequence, under the doubling property and the sub-Gaussian heat kernel estimate, the Riesz transform is weak $(1,1)$ hence $L^p$ bounded for $p\in (1,2)$  by interpolation with \eqref{2}. Section \ref{Rieszgraph} is the counterpart of Sections \ref{integratestim} and \ref{weak11} on graphs.   
Finally, in Section \ref{Vicsek}, we prove that on Vicsek manifolds, for all $p\in (1,2)$, \eqref{RRp} does not hold. It follows that on such manifolds, in  the range $1<p<+\infty$, \eqref{Rp} holds if and only if $1<p\leq 2$ and \eqref{RRp} holds if and only if $2\le p<+\infty$.

\section{Integrated estimate for the gradient of the heat kernel} \label{integratestim}

In the Gaussian case, the estimates from this section are due to Grigor'yan  in \cite{Gr95}. They go through a version of Lemma \ref{weightednablap} for $q=2$, which follows easily from the Gaussian estimate of the heat kernel by integration by parts
(see  also \cite[Lemma 2.3]{CD99}  for a simpler version).  In the subgaussian case, the corresponding argument  yields a factor $t^{-\alpha}$ with $0<\alpha<1/2$  in the right hand side of the crucial estimate \eqref{int gradient}  below,  which is not enough to prove the $L^p$-boundedness of the Riesz transform, see \cite[Lemma 3.2]{Ch15}. To obtain the stronger estimate we need, namely \eqref{intenew},  we will choose $1<q<2$, which will enable us  to  use a powerful argument by Stein.

Let us first derive some easy consequences of the doubling
property and the sub-Gaussian heat kernel estimate. 

 A classical consequence of (\ref{doubling}) is that there exists $\nu >0$ such that
\begin{equation} \label{D1}
\frac{V(x,r)}{V(x,s)}\lesssim \left(\frac{r}{s}\right)^\nu,
\,\,\forall x\in M, \,r\geq s>0.
\end{equation}

Fix now once and for all  $m>2$.

\begin{lemma}\label{integrated hk}
Let $M$ be a complete  Riemannian manifold satisfying the volume doubling property (\ref{doubling}) and the sub-Gaussian heat kernel upper bound (\ref{B}). For any $q > 1$, there holds, for all $y\in M$,
\begin{equation}\label{Lp hk}
\int_{M}h_t^q(x,y) d\mu(x)  \lesssim  
\left\{ \begin{aligned}
         & \frac{1}{\left[V(y,\sqrt t)\right]^{q-1}} ,&0<t<1, \\
         & \frac{1}{\left[V(y,t^{1/m})\right]^{q-1}}, &t\geq 1, 
 \end{aligned}\right.
\end{equation}
and 
\begin{equation}\label{analyticity}
\int_{M} \abs{\Delta h_t(x,y)}^q d\mu(x)  \lesssim 
\left\{ \begin{aligned}
         &  \frac{1}{t^q\left[V(y,\sqrt t)\right]^{q-1}}, &0<t<1, \\
         &  \frac{1}{t^q \left[V(y,t^{1/m})\right]^{q-1}}, &t\geq 1. 
 \end{aligned}\right.
\end{equation}
\end{lemma}

\begin{proof}
It follows from (\ref{doubling}) that for any $r\ge 0$ and any $c>0$ (see for example Step 1 of the proof of Lemma 3.2 in \cite{Ch15}), 
\begin{equation}
\label{est}
\int_{d(x,y)>r} e^{-c\frac{d^2(x,y)}{t}}d\mu(x)
 \lesssim
 e^{-\frac{c}{2}\frac{r^2}{t}}V\br{y,t^{1/2}},
\end{equation}
and
\begin{equation}
\label{est2}
\int_{d(x,y)>r} e^{-c\br{\frac{d^m(x,y)}{t}}^{1/(m-1)}}d\mu(x)
 \lesssim
 e^{-\frac{c}{2}\br{\frac{r^m}{t}}^{1/(m-1)}}V\br{y,t^{1/m}}.
\end{equation}

Thus \eqref{B} yields that for $t\geq 1$,
\begin{eqnarray*}
\int_{M} h_t^q(x,y) d\mu(x) 
& \lesssim & 
\frac{1}{\left[V(y,t^{1/m})\right]^{q}} \int_{M} e^{-cq\br{\frac{d^m(x,y)}{t}}^{1/(m-1)}} d\mu(x)
\\ &\lesssim&
\frac{1}{\left[V(y,t^{1/m})\right]^{q-1}},
\end{eqnarray*}
 with an analogous argument for $0<t<1$, hence \eqref{Lp hk}.

Now,  by the analyticity of the heat semigroup on $L^q(M,\mu)$ for $1<q<+\infty$,  for  $t\geq 2$,
\begin{eqnarray*}
\int_{M} \abs{\Delta h_t(x,y)}^q d\mu(x) 
&\leq& 
\norm{\Delta e^{-\frac{t}{2}\Delta}}_{q\to q}^q \int_M \abs{h_{t/2}(x,y)}^q d\mu(x)
\\ &\lesssim&
\frac{1}{t^q\left[V(y,t^{1/m})\right]^{q-1}}.
\end{eqnarray*}
Similarly for $0<t<2$, we repeat the proof above and obtain that
\[
\int_{M} \abs{\Delta h_t(x,y)}^q d\mu(x) \lesssim \frac{1}{t^q \left[V(y,t^{1/2})\right]^{q-1}}.
\]
 Since, for $1<t<2$, it holds that $V(y,t^{1/2})\simeq V(y,t^{1/m})$, we obtain \eqref{analyticity}.
\end{proof}

\medskip

\begin{lemma} \label{weightednablap} 
Let $M$ be a complete  Riemannian manifold satisfying the volume doubling property (\ref{doubling}) and the sub-Gaussian heat kernel upper 
bound (\ref{B}).
For  $q\in (1,2)$,  there exists $c>0$ such that, for all $y\in M$ and all $0<t< 1$, one has
$$
 \left\Vert \left\vert \nabla h_t(\cdot,y)\right\vert \exp\left(c\frac{d^2(\cdot,y)}t\right)\right\Vert_q\lesssim  \frac 1{\sqrt{t}\left[V(y,t^{1/2})\right]^{1-\frac 1q}}, 
$$
 and for all $y\in M$ and all $t\geq 1$, one has
$$
 \left\Vert \left\vert \nabla h_t(\cdot,y)\right\vert \exp\left(c\left(\frac{d^m(\cdot,y)}t\right)^{1/(m-1)}\right)\right\Vert_q\lesssim  \frac 1{\sqrt{t}\left[V(y,t^{1/m})\right]^{1-\frac 1q}}. 
$$
\end{lemma}

\begin{proof} Let  $q\in (1,2)$.  Fix  $y\in M$. Define, for all $x\in M$ and all $t>0$, $u(x,t):=h_t(x,y)$. Since $M$ is connected, $u(x,t)$ is positive. 
A straightforward computation (see \cite[Lemma 1, page 86]{St70}, \cite[Lemma 2, page 49]{St701} in different settings) shows that
$$
  \br{\frac{\partial}{\partial t}+\Delta} u^q(x,t) 
=
q u^{q-1}(x,t) \br{\frac{\partial}{\partial t} +\Delta} u(x,t) - q(q-1)u^{q-2}(x,t)|\nabla_x u(x,t)|^{2}. 
$$
Since $u$ is a solution to the heat equation,
$$
 \br{\frac{\partial}{\partial t}+\Delta} u^q(x,t) 
=
- q(q-1)u^{q-2}(x,t)|\nabla_x u(x,t)|^{2}, 
$$
hence
\begin{equation}\label{eqj}
 |\nabla_x u(x,t)|^{2}=\frac{1}{q(q-1)}u^{2-q}(x,t)J(x,t), 
\end{equation}
where $ J(x,t)= -\br{\frac{\partial}{\partial t}+\Delta}u^q(x,t)$.  Note in particular that $J(x,t)\geq 0$ for all $x\in M$ and all $t>0$. 

 Since $u(.,t) \in L^q$,  then  $u^q(.,t) \in L^1$  and   $\int_M \Delta u^q(x,t)d\mu(x)=0$. 
Together with H\"older inequality, we have
\begin{eqnarray*}
\int_M J(x,t) \,d\mu(x)
& =& 
 -\int_M \frac{\partial}{\partial t} u^q(x,t) \,d\mu(x)
\leq 
\int_M |q u^{q-1}(x,t) \Delta u(x,t)| \,d\mu(x)
\\ &\lesssim&   
 \br{\int_M u^q(x,t)\,d\mu(x)}^{\frac{q-1} q} \br{\int_M |\Delta u(x,t)|^q d\mu(x)}^{1/q}. 
\end{eqnarray*}

From now on, let us assume that $t\ge 1$. The proof for $0<t<1$ is similar and we omit it here. It follows easily from  Lemma \ref{integrated hk} that
\begin{equation}
\label{JL1}
 \int_M J(x,t)\,d\mu(x)
 \lesssim \frac 1{t\left[V(y,t^{1/m})\right]^{q-1}}. 
\end{equation}

According to \eqref{eqj}, for all $x\in M$ and $t>0$,
$$
\left\vert \nabla_x u(x,t)\right\vert \exp\left(c\left(\frac{d^m(x,y)}t\right)^{1/(m-1)}\right)=C_q\left[u(x,t)\right]^{1-\frac q 2}J^{1/2}(x,t)\exp\left(c\left(\frac{d^m(\cdot,y)}t\right)^{1/(m-1)}\right), 
$$
 so that
\begin{eqnarray*}
&&  \left\Vert \left\vert \nabla u(\cdot,t)\right\vert \exp\left(c\left(\frac{d^m(\cdot,y)}t\right)^{1/(m-1)}\right)\right\Vert_q^q   \\
 &=&  \displaystyle C_q^q \int_M \left[u(x,t)\right]^{q\left(1-\frac q 2\right)}  J^{q/2}(x,t)\exp\left(cq\left(\frac{d^m(x,y)}t\right)^{1/(m-1)}\right)\,d\mu(x). 
\end{eqnarray*}

By H\"older, since  $q<2$,   the latter quantity is bounded from above by
 $$  \left(\int_M u^{q}(x,t)\exp\left(c\frac{q}{1-\frac q 2}\left(\frac{d^m(x,y)}t\right)^{1/(m-1)}\right)\,d\mu(x)\right)^{1-\frac q2}\left(\int_M J(x,t)\,d\mu(x)\right)^{\frac q2}. $$
 The second factor is bounded from above by \eqref{JL1}.
 As for the first one, one uses \eqref{B}, so that if $c$ is small enough (only depending on $q$), there exists $c'>0$ such that
 \begin{eqnarray}
&&\left(\int_M u^{q}(x,t)\exp\left(c\frac{q}{1-\frac q 2}\left(\frac{d^m(x,y)}t\right)^{1/(m-1)}\right)\,d\mu(x)\right)^{1-\frac q2}\\ \nonumber
&\lesssim&  \frac 1{\left[V(y,t^{1/m})\right]^{q(1-\frac q2)}}\left(\int_M \exp\left(-c^{\prime}\left(\frac{d^m(x,y)}t\right)^{1/(m-1)}\right)\,d\mu(x)\right)^{1-\frac q2} \\ \nonumber
& \lesssim &  \frac 1{\left[V(y,t^{1/m})\right]^{(q-1)(1-\frac q2)}},
\end{eqnarray}
where the last line is due to \eqref{est2} with $r=0$. 
It follows that
 
$$  \left\Vert \left\vert \nabla u(\cdot,t)\right\vert \exp\left(c\left(\frac{d^m(\cdot,y)}t\right)^{1/(m-1)}\right)\right\Vert_q^q \lesssim \frac 1{t^{q/2}[V(y,t^{1/m})]^{q-1}} ,$$ 
which yields the conclusion.
\end{proof}

\begin{corollary}
Let $M$ be a complete  Riemannian manifold satisfying (\ref{doubling}) and (\ref{B}). Then there exists $c>0$ such that for all $y\in M$, all $0<t< 1$ and all $r>0$,
\begin{equation}\label{int gradient'}
\int_{M\setminus B(y,r)} \left\vert \nabla_xh_t(x,y)\right\vert \,d\mu(x)\lesssim \frac 1{\sqrt{t}} \exp\left(-c\frac{r^2}t\right),
\end{equation}
and for all $y\in M$, all $t\geq 1$ and all $r>0$,
\begin{equation}\label{int gradient}
\int_{M\setminus B(y,r)} \left\vert \nabla_xh_t(x,y)\right\vert \,d\mu(x)\lesssim \frac 1{\sqrt{t}} \exp\left(-c\left(\frac{r^m}t\right)^{1/(m-1)}\right).
\end{equation}
\end{corollary}

\begin{proof} Fix  $q\in (1,2)$  and let $c>0$ be given by Lemma \ref{weightednablap}. Then, by H\"older,
\begin{eqnarray*}
 \int_{M\setminus B(y,r)} \left\vert \nabla_xh_t(x,y)\right\vert \,d\mu(x) & \leq &  \displaystyle \left\Vert \left\vert \nabla h_t(\cdot,y)\right\vert \exp\left(c\left(\frac{d^m(\cdot,y)}t\right)^{1/(m-1)}\right)\right\Vert_q\\
& &  \displaystyle \left(\int_{M\setminus B(y,r)}  \exp\left(-c^{\prime}\left(\frac{d^m(x,y)}t\right)^{1/(m-1)}\right)\,d\mu(x)\right)^{1-\frac 1 q},\end{eqnarray*}
and by \eqref{est2} and Lemma \ref{weightednablap}, for $t\ge 1$,
$$ \int_{M\setminus B(y,r)} \left\vert \nabla_xh_t(x,y)\right\vert \,d\mu(x) \lesssim  \displaystyle \frac 1{\sqrt{t}\left[V(y,t^{1/m})\right]^{1-\frac 1 q}} \exp\left(-c''\left(\frac{r^m}t\right)^{1/(m-1)}\right)\left[V(y,t^{1/m})\right]^{1-\frac 1 q}. 
$$
which gives \eqref{int gradient}.

For $0<t<1$, \eqref{int gradient'} is obtained in the same way.
\end{proof}

\begin{corollary} \label{EstimateKernelCor}
Let $M$ be a complete  Riemannian manifold satisfying (\ref{doubling}) and  (\ref{B}).  
Then for all $y\in M$, all $r,t>0$,
\begin{equation}\label{intenew}
\int_{M\setminus B(y,r)} \left\vert \nabla_x h_t(x,y)\right\vert \,d\mu(x)\lesssim \frac 1{\sqrt{t}} \exp\left(-c\left(\frac{s}t\right)^\frac{1}{m-1}\right),
\end{equation}
where $s=r^2$ if $r<1$ and $s=r^m$ if $r\geq 1$.
\end{corollary}

\begin{proof}
Set
$$I(t,r) := \sqrt{t} \int_{M\setminus B(y,r)} \left\vert \nabla_x h_t(x,y)\right\vert \,d\mu(x).$$
\begin{itemize}
\item Let $r,t<1$, then \eqref{int gradient'} yields 
$$I(t,s) \lesssim e^{-c\frac{r^2}{t}} = e^{-c\frac{s}{t}}.$$
\item Let $t<1\leq r$, then \eqref{int gradient'} yields $I(t,s) \lesssim  e^{-c\frac{r^2}{t}}$. 
Yet, since $t<r$, we have $\frac{r^2}{t} \geq \left(\frac{r^m}{t}\right)^{1/(m-1)}$. Therefore
$$I(t,s) \lesssim e^{-c(\frac{s}{t})^{1/(m-1)}}.$$
\item Let $r<1\leq t$, then \eqref{int gradient} yields $I(t,s) \lesssim e^{-c(\frac{r^m}{t})^{1/(m-1)}}$.
Yet, since $r<t$, we have $\left(\frac{r^m}{t}\right)^{1/(m-1)} \geq \frac{r^2}{t}$. Therefore
$$I(t,s) \lesssim e^{-c\frac{s}{t}}.$$
\item Let $r,t\geq 1$, then \eqref{int gradient} yields 
$$I(t,s) \lesssim e^{-c(\frac{r^m}{t})^{1/(m-1)}} = e^{-c(\frac{s}{t})^{1/(m-1)}}.$$
\end{itemize}
The conclusion follows from the fact that, for all $x>0$,
$$e^{-cx^\gamma} \lesssim e^{-cx^{\beta}}$$
if $0<\beta <\gamma$.
\end{proof}


\section{Weak (1,1) boundedness of the Riesz transform} \label{weak11}

We will prove the weak $(1,1)$ boundedness of the Riesz transform by adapting the argument of the second author and Duong (\cite{CD99}), or rather the part of the argument that is directly inspired from \cite{DM99}, to the sub-Gaussian case as in \cite{ChThese}. This adaptation is straighforward, which confirms that the main novelty of the present paper is the use of Stein's argument in Section \ref{integratestim} to derive \eqref{intenew}. We give it however for the sake of completeness.

 \begin{remark}
One could also use a theorem on Calder\'on-Zygmund operators without kernel as written in \cite{BK03} and improved in \cite[Theorem 1.1]{Aus07}. However, since \cite[Theorem 1.1]{Aus07} is not exactly written in the form we need, we choose to follow the original method of the second author and Duong.
\end{remark}

First recall the Calder\'on-Zygmund decomposition (see for example  \cite[Corollaire 2.3]{CW71}):

\begin{theorem}\label{C-Z}
Let $(M,d,\mu)$ be a metric measured space satisfying \eqref{doubling}. Then there exists $C>0$ such that for any given function $f\in L^1(M)\cap 
L^2(M)$ and $\lambda >0$, there exists a decomposition of $f$, $\displaystyle f=g+b=g+\sum_{i\in I} b_i $ so that
\begin{enumerate}
\item $|g(x)|\leq C\lambda$  for almost all $x\in M$;

\item There exists a sequence of balls $B_i =B(x_i,r_i)$ so that,   for all $i\in I$,  $b_i\in L^1(M)\cap L^2(M)$  is supported in $B_i$ and
\[
\int| b_i(x)|d\mu(x)\leq C\lambda\mu(B_i)\text{ and } 
\int b_i(x)d\mu(x)=0; 
\]

\item $\displaystyle \sum_{i\in I} \mu(B_i)\leq \frac{C}{\lambda}\int|f(x)|d\mu(x)$;

\item There exists $k\in \mathbb{N}^*$ such that each $x\in M$ is contained in at most $k$ balls $B_i$.
\end{enumerate}
\end{theorem}

Note that it follows from $(2)$ and $(3)$ that  $\Vert b\Vert_1\leq C\Vert f\Vert_1$ and $\Vert g\Vert_1\leq(1+C)\Vert f\Vert_1$.

\noindent{\em Proof of Theorem \ref{main}:}

\noindent Denote $T=\nabla \Delta^{-1/2}$. Since $T$ is always $L^2$ bounded, it is enough to show that $T$ is weak $(1,1)$. Then, for $1<p<2$, the $L^p$ 
boundedness of $T$ follows from the Marcinkiewicz interpolation theorem.

We aim to show that 
\[
\mu(\{x: |T f(x)|>\lambda\})\leq C\lambda^{-1}\Vert f\Vert_1,
\]
for all $\lambda>0$  and all $f\in L^1(M)\cap L^2(M)$. 

Let $\lambda>0$ and $f\in L^1(M)\cap L^2(M)$.  Consider the Calder\'{o}n-Zygmund decomposition of $f$ at the level $\lambda$   (we shall use the notation of Theorem \ref{C-Z} in the sequel of the argument). Then
\[
\mu(\{x: |T f(x)|>\lambda\})
\leq\mu(\{x:|T g(x)|> \lambda/ 2 \})+\mu(\{x:|T b(x)|>\lambda/ 2 \}).
\]
Since $T$ is $L^2$ bounded and $|g(x)|\leq C\lambda$, it follows that
\[
\mu(\{x: |T g(x)|>\lambda/ 2\}) 
\leq C\lambda^{-2} \norm{g}_2^2 \leq C\lambda^{-1}\norm{g}_1
\leq C\lambda^{-1}\Vert f\Vert_1.
\]

It remains to prove 
\begin{equation} \label{secondpart}
\mu\left(\left\{x: \left|T \left(\sum_{i\in I}b_i\right)(x)\right|>\lambda/ 2\right\}\right) \leq C \lambda^{-1}\Vert f\Vert_1.
\end{equation}
Define for convenience
\begin{eqnarray}
\label{t} \rho(r)=\left\{ \begin{aligned}
         &r^{2},&0<r<1, \\
         & r^{m},&r\geq 1. 
         \end{aligned}\right.
\end{eqnarray} 
For each $i\in I$, write 
\begin{equation}\label{divide}
T b_i=T e^{-t_i\Delta }b_i+T \left(I-e^{-t_i\Delta }\right)b_i,
\end{equation}
where $t_i=\rho(r_i)$. 

We have then
\begin{eqnarray*}
\mu\left(\left\{x: \left|T \left(\sum_{i\in I}b_i\right) (x)\right|>\lambda/ 2\right\}\right) 
& \leq & \mu\left(\left\{x: \left|T \left(\sum_{i\in I}e^{-t_i\Delta }b_i\right) (x)\right|>\lambda/ 4\right\}\right)
\\ &&+ \mu\left(\left\{x: \left|T \left(\sum_{i\in I}(I-e^{-t_i\Delta })b_i\right) (x)\right|>\lambda/ 4\right\}\right).
\end{eqnarray*}

We begin with the estimate of the first term. Since $T$ is $L^2$ bounded, then
\[
\mu\left(\left\{x: \left|T \left(\sum_{i\in I}e^{-t_i\Delta }b_i\right) (x)\right|>\lambda/ 4\right\}\right)
\leq \frac{C}{\lambda^2}\left\Vert \sum_{i\in I}e^{-t_i\Delta }b_i\right\Vert_2^2.
\]
By a duality argument, 
\begin{eqnarray*}
\norm{\sum_{i\in I}e^{-t_i\Delta }b_i}_2 
&=& 
\sup_{\norm{\phi}_2=1} \abs{<\sum_{i\in I}e^{-t_i\Delta }b_i,\phi>}
\\&=& 
\sup_{\norm{\phi}_2=1}\abs{\sum_{i\in I} <b_i, e^{-t_i\Delta } \phi>}
\\&\leq & 
\sup_{\norm{\phi}_2=1}  \sum_{i\in I} \Big(\sup_{y\in B_i} e^{-t_i\Delta }|\phi|(y)\Big) \int_M |b_i| d\mu .
\end{eqnarray*}
 We claim that for all $r,t>0$ such that $t=\rho(r)$ and every   measurable function  $\varphi\geq 0$
\begin{equation} \label{calculusduality}
\sup_{y\in B(x,r)} e^{-t\Delta } \varphi(y) \leq C \inf_{z\in B(x,r)}  \mathcal M\varphi (z),
\end{equation}
where $\M$ denotes the uncentered Littlewood-Paley maximal operator: 
\[
\M f(x)=\sup_{B\ni x} \frac{1}{\mu(B)} \int_B |f(y)| d\mu(y).
\]
Estimates of the type  \eqref{calculusduality} are classical (see for instance  \cite[Proposition 2.4]{DR}). Let us give a proof of this particular instance for the sake of completeness. 
We use the notation $C_j(B)$ for $2^{j+1}B \setminus 2^jB$ when $j\geq 2$ and $C_1(B) = 4B$.
If $B=B(x,r)$ and $t=\rho(r)$, for any $y\in B$ and any $z\in C_j(B)$, we have with  \eqref{B} 
$$h_t(y,z) \lesssim  \frac{1}{V(y,r)}e^{-c2^j}.$$
Therefore
\begin{eqnarray*}
e^{-t\Delta } \varphi(y)
&=& 
\int_{M}h_{t}(y,z) \varphi(z) d\mu(z)
=
\sum_{j=1}^{\infty} \int_{C_j(B)}h_{t}(y,z) \varphi(z) d\mu(z)
\\ & \lesssim & 
\sum_{j=1}^{\infty} \frac{V(x,2^{j+1}r)}{V(y,r)} e^{-c2^j} \frac{1}{V(x,2^{j+1}r)}\int_{B(x,2^{j+1}r)} \varphi(z) d\mu(z)
\\ & \lesssim & 
 \sum_{j=1}^{\infty}  \frac{V(y,2^{j+2}r)}{V(y,r)} e^{-c2^j} \frac{1}{V(2^{j+1}B)}\int_{2^{j+1}B} \varphi(z) d\mu(z)
\\ & \lesssim & 
 \inf_{z\in B} \mathcal M\varphi(z),
\end{eqnarray*}
which is exactly \eqref{calculusduality}.

\noindent Then, if $\|\phi\|_2 = 1$,
\begin{eqnarray*}
\sum_{i\in I} \Big(\sup_{y\in B_i} e^{-t_i\Delta }|\phi|(y)\Big) \int_M |b_i| d\mu 
&\lesssim&
\sum_{ i\in I } \inf_{y\in B(x,t^{1/m})} {\mathcal M\phi (y)} \int_M |b_i| d\mu
\\ &\lesssim&  
\lambda \sum_{ i\in I} \int_{B_i}  {\M\phi(y)} d\mu(y) 
\\ & =  & 
\lambda \int_M \sum_{ i\in I } \chi_{B_i}(y) {\M\phi(y)} d\mu(y)
\\ &\lesssim & 
\lambda \int_{\cup_{ i\in I}B_i}  {\M\phi(y)} d\mu(y)
\\ &\lesssim&
\lambda \left[\mu\left(\bigcup_{ i\in I}B_i\right)\right]^{1/2} \| {\M\phi}\|_2 
\\ &\lesssim&
\lambda^{1/2}\norm{f}_1^{1/2}.
\end{eqnarray*}
The second line follows from property $(2)$ of the Calder\'on-Zygmund decomposition.  The fourth  line  is due to the finite overlapping of $\{B_i\}$. 
The last but one is a consequence of H\"older inequality and the last one is due to the $L^2$-boundedness of the Hardy-Littlewood maximal function and property $(3)$ of the Calder\'on-Zygmund decomposition.

Hence
\[
\norm{\sum_{i\in I}e^{-t_i\Delta }b_i}_2  \lesssim  \lambda^{1/2}\norm{f}_1^{1/2},
\]
and
\[
\mu\left(\left\{x: \left|T  \left(\sum_{i\in I}e^{-t_i\Delta }b_i\right)(x)\right|>\lambda/ 4\right\}\right)  \lesssim  \lambda^{-1}\norm{f}_1.
\]

It remains to show that
$\mu\left(\left\{x: \left|T \left(\sum_{i\in I}(I-e^{-t_i\Delta })b_i\right)(x)\right|>\lambda/ 4\right\}\right)\leq C \lambda^{-1} \Vert f\Vert_1$. 
Write
\begin{eqnarray*}
&&\mu\left(\left\{x: \left|T \left(\sum_{i\in I}(I-e^{-t_i\Delta })b_i\right)(x)\right|>\lambda /4\right\}\right)
\\&\leq & 
\mu\left(\left\{x\in \bigcup_{i\in I}2B_i: \left|T \left(\sum_{i\in I}(I-e^{-t_i\Delta })b_i\right)(x)\right|>\lambda /4\right\}\right)
\\&& 
+\mu\left(\left\{x\in M\setminus \bigcup_{i\in I}2B_i: \left|T  \left(\sum_{i\in I}(I-e^{-t_i\Delta })b_i\right)(x)\right|> \lambda /4 \right\}\right)
\\&\leq & 
\sum_{i\in I}\mu (2B_i)+\frac{4}{\lambda} \sum_{i\in I}
\int_{M\backslash 2B_i}|T (I-e^{-t_i\Delta })b_i(x)|d\mu(x).
\end{eqnarray*}
We claim that for every $i\in I$, it holds 
\begin{equation} \label{claim}
\int_{M\setminus 2B_i}|T(I-e^{-t_i\Delta })b_i(x)|d\mu(x)\leq C \Vert b_i\Vert_1.
\end{equation}
Therefore by the doubling property and the Calder\'on-Zygmund decomposition, 
\[
\mu\left(\left\{x: \left|T \left(\sum_{i\in \mathcal I}(I-e^{-t_i\Delta })b_i\right)(x)\right|> \lambda/ 4 \right\}\right)
 \lesssim
 \sum_{ i\in I}\mu (B_i)+ \frac{1}{\lambda} \sum_{i\in I} \norm{b_i}_1
 \lesssim \lambda^{-1} \Vert f\Vert_1,
\]
which concludes the proof.

So it remains to prove \eqref{claim}. First, the spectral theorem gives us that $\Delta ^{-1/2}f=c\int_0^\infty e^{-s\Delta}f\frac{ds}{\sqrt s}$. 
Therefore, if $f\in L^2(M)$, 
\[\begin{split} 
\Delta^{-1/2} (I-e^{-t\Delta}) f &=c\int_0^\infty (e^{-s\Delta} - e^{-(s+t)\Delta})f\frac{ds}{\sqrt s} \\
& = c\int_{0}^\infty \left( \frac{1}{\sqrt s} - \frac{\chi_{\{s>t\}}}{\sqrt{s-t}}\right) e^{-s\Delta} f \, ds.
\end{split} \]
and thus for all $x\in M$
\[\begin{split}
|T(I-e^{-t\Delta}) f(x)| \leq C \int_M |f(y)| d\mu(y) \int_0^\infty \left| \frac{1}{\sqrt s} - \frac{\chi_{\{s>t\}}}{\sqrt{s-t}}\right| \left|\nabla_x  h_s(x,y)\right|ds.  
\end{split}\]
Set 
$$k_t(x,y) = \int_0^\infty \left| \frac{1}{\sqrt s} - \frac{\chi_{\{s>t\}}}{\sqrt{s-t}}\right| |\nabla_x  h_s(x,y)| ds.$$

Since $b_i\in L^2(M)$,  we have
\begin{eqnarray*}
\int_{M\backslash 2B_i}|T(I-e^{-t_i\Delta })b_i(x)|d\mu(x)
&\lesssim & 
\int_{M\backslash 2B_i}\int_{B_i}k_{t_i}(x,y)|b_i(y)|d\mu(y)d\mu(x)
\\&\lesssim &
\int_M |b_i(y)|\int_{d(x,y)\geq r_i}k_{t_i}(x,y)d\mu(x)d\mu(y).
\end{eqnarray*}

The claim \eqref{claim} will be proven if $\int_{d(x,y)\geq r}k_t(x,y) d\mu(x)$ is uniformly bounded for $t=\rho(r) >0$.
By  Corollary \ref{EstimateKernelCor}, we have
\begin{eqnarray*}
&&\int_{d(x,y)\geq r} k_t(x,y) d\mu(x)
\\&=&
\int_{d(x,y)\geq r}\int_0^\infty \left|\frac{1}{\sqrt s}-\frac{1_{\{s>t \}}}{\sqrt{s-t}}\right| |\nabla_x  h_s(x,y)|ds \,d\mu(x)
\\& = &
\int_0^\infty \abs{\frac{1}{\sqrt s}-\frac{1_{\{s>t \}}}{\sqrt{s-t}}}
\cdot \int_{d(x,y)\geq r}\abs{\nabla_x  h_s(x,y)}d\mu(x)ds
\\&\lesssim & 
\int_0^\infty \abs{\frac{1}{\sqrt s}-\frac{1_{\{s>t \}}}{\sqrt{s-t}}}e^{-c(\frac{t}{s})^{\frac 1 {m-1}}} \frac{ds}{\sqrt s}
\\&=& 
\int_0^t e^{-c\left(\frac{t}{s}\right)^{\frac 1 {m-1}}} \frac{ds}{s}
+
\int_t^\infty \abs{\frac{1}{\sqrt s}-\frac{1}{\sqrt{s-t}}} e^{-c\left(\frac{t}{s}\right)^{\frac 1 {m-1}}} \frac{ds}{\sqrt s}
\\&:=& 
K_1+K_2.
\end{eqnarray*}

 It is easily checked that $K_1,K_2$ are  bounded uniformly in $t>0$. Indeed, 
 \[
K_1= \int_0^1 e^{-c\left(\frac{1}{u}\right)^{\frac 1 {m-1}}} \frac{du}{u}<+\infty
\]
and
\begin{eqnarray*}
K_2
&\leq &
\int_t^\infty \abs{\frac{1}{\sqrt s}-\frac{1}{\sqrt{s-t}}} \frac{ds}{\sqrt s}
\\ & =&
\int_0^\infty \abs{\frac 1{\sqrt {u+1}}-\frac {1}{\sqrt u}} \frac{du}{\sqrt {u+1}}<+\infty. 
\end{eqnarray*}
\medskip
Note that we get the second line by the change of variable  $u=\frac{s}{t}-1$.


\section{The case of graphs} \label{Rieszgraph}

In this section, we give the counterpart of Theorem \ref{main} on graphs. Let us begin with some definitions.

A (weighted unoriented) graph is defined by a couple $(\Gamma,\mu)$ where $\Gamma$ is an infinite countable set (the set of vertices) and $\mu \geq 0$ is a symmetric weight on $\Gamma \times \Gamma$. We define the set of edges as $E = \{(x,y) \in \Gamma\times \Gamma, \, \mu_{xy} >0\}$ and we write $x\sim y$ (we say that $x$ and $y$ are neighbours) whenever $(x,y)\in E$.

\noindent We assume that the graph is connected and locally uniformly finite. A graph is connected if for all $x,y\in \Gamma$, there exists a path joining $x$ and $y$, that is a sequence $x=x_0,\dots,x_n = y$ such that for all $i\in \llbracket 1,n \rrbracket$, $x_{i-1}\sim x_i$ (the length of such a path being $n$). A graph is locally uniformly finite if there exists $M_0\in \N$ such that, for all $x\in \Gamma$, $\#\{y\in \Gamma, \, y\sim x\} \leq M_0$.

\noindent The graph is endowed with its natural metric $d$, where, for all $x,y\in \Gamma$, $d(x,y)$ is the length of the shortest path joining $x$ and $y$. 

\noindent We define the weight $m(x)$ of a vertex $\Gamma$ by $m(x) = \sum_{y\sim x} \mu_{xy}$. More generally, the measure of a subset $E\subset \Gamma$ is defined as $m(E): = \sum_{x\in E} m(x)$. For $x\in \Gamma$ and $r>0$, we denote by $B(x,r)$ the open ball of center $x$ and radius $r$ and by $V(x,r)$ the measure (or volume) of $B(x,r)$.

In this situation, we recall the volume doubling property. The graph $(\Gamma,\mu)$ satisfies \eqref{doublingraph} if, for any $x\in \Gamma$ and any $r>0$,
\begin{equation} \label{doublingraph} \tag{$D_{\Gamma}$}
V(x,2r) \lesssim  V(x,r).
\end{equation}

 For all $x,y \in \Gamma$, the reversible Markov kernel $p(x,y)$ is defined as $p(x,y) = \frac{\mu_{xy}}{m(x)}$. The kernel $p_k(x,y)$ is then defined recursively for all $k\in \N$ by
$$\left\{ \begin{array}{l}  p_0(x,y) = \delta(x,y) \\  p_{k+1}(x,y) = \sum_{z\in \Gamma} p(x,z)p_k(z,y). \end{array} \right.$$
Note that, for all $x\in \Gamma$ and all $k\geq 0$,  $\sum_{y\in \Gamma} p_k(x,y) = 1$. Moreover, one has $p(x,y) m(x) = p(y,x) m(y)$ for all $x,y\in \Gamma$.
The operator $P$ has kernel $p$, meaning that the following formula holds for all  $x\in\Gamma$: $P f(x) = \sum_{y\in \Gamma} p(x,y) f(y)$. For all $k\geq 1$, $P^k$ has kernel $p_k$.

A second assumption will be used. We say that $\Gamma$ satisfies \eqref{LB} if there exists $\varepsilon >0$ such that, for all $x\in \Gamma$,
\begin{equation} \label{LB} \tag{$LB$}
p(x,x) > \varepsilon.
\end{equation}
 The assumption \eqref{LB} implies the $L^2$-analyticity of $P$, that is $\|(I-P)P^n\|_{2\to 2} \leq \frac{C}{n}$, for all $n\in\N^*$, but the converse is false. We refer to \cite[Theorem 1.9]{FeSpectre} for a way to recover \eqref{LB} from the $L^2$-analyticity of $P$.
 
We say that $(\Gamma, \mu)$ satisfies the (Markov kernel) estimates \eqref{UEm} with $m\geq2$ if for all $x,y\in \Gamma$ and $k\in \N^*$
\begin{equation} \tag{$U\!E_{m,\Gamma}$} \label{UEm}
p_{k-1}(x,y) \lesssim \frac{m(y)}{V(y,k^{1/m})} \exp\left(-c \left(\frac{d^m(x,y)}{k} \right)^{\frac1{m-1}} \right).
\end{equation}
The shift between $k$ and $k-1$ allows  \eqref{UEm} to cover the obvious estimate of  $p_0(x,y)$.
When $m=2$, the above estimates are called Gaussian estimates. When $m>2$, they are called sub-Gaussian estimates.

We use the random walk $P$ to define a positive Laplacian and a length of the gradient by
$$\Delta_{\Gamma} : = I-P$$
and
$$\nabla_{\Gamma} g(x) = \left(\frac12 \sum_{y\in \Gamma} p(x,y) |g(y)-g(x)|^2 \right)^{1/2},$$
 say  for $g\in c_0(\Gamma)$, where $c_0(\Gamma)$ denotes the space of   finitely supported functions  on $\Gamma$.
In this context, the Riesz transform is the sublinear operator $\nabla_{\Gamma}\Delta_{\Gamma}^{-1/2}$. 
As in the case of Riemannian manifolds, we can study the $L^p$ boundedness of the Riesz transform, that is we can wonder when
\begin{equation} \tag{$R_{p,\Gamma}$} \label{Rpg} \|\nabla_{\Gamma} g\|_p \lesssim  \|\Delta_{\Gamma}^{1/2} g\|_p,\ \forall\,g\in
c_0(\Gamma).\end{equation}

The counterpart on graphs of Theorem \ref{thm-cd} is a result of the fourth author:

\begin{theorem}\cite[Theorem 1]{Russ00}
Let $(\Gamma,\mu)$ be a weighted graph satisfying \eqref{LB}, \eqref{doublingraph} and $(U\!E_{2,\Gamma})$. 
Then the Riesz transform $\nabla_{\Gamma} \Delta_{\Gamma}^{-1/2}$ is weak $(1,1)$ and $L^p$ bounded for all $p\in (1,2]$.
\end{theorem}

As in the case of Riemannian manifolds, we can extend the $L^p$-boundedness of the Riesz transform for $1<p\leq 2$ when $(\Gamma,\mu)$ satisfies sub-Gaussian estimates.

\begin{theorem} \label{maingraph}
Let $(\Gamma,\mu)$ be a weighted graph satisfying \eqref{LB}, \eqref{doublingraph} and \eqref{UEm} for some $m>2$. 
Then the Riesz transform $\nabla_{\Gamma} \Delta_{\Gamma}^{-1/2}$ is weak $(1,1)$ and $L^p$ bounded for all $p\in (1,2]$.
\end{theorem}

Recall that the statement on $L^p$ in Theorem \ref{maingraph} also follows from  \cite[Corollary 1.30]{F15}. The proof of Theorem \ref{maingraph} is similar to the one of Theorem \ref{main} and we will only prove the counterpart of Lemma \ref{weightednablap}, where  the main difficulty lies.
Indeed, the ``Hardy-Stein'' identity \eqref{eqj} doesn't hold on graphs and a trick of Dungey (in \cite{Du08}), improved by the third author in \cite[Section 4]{Fen15}, is needed.

\begin{lemma} \label{lemmagraph}
Let $(\Gamma,\mu)$ be a weighted graph satisfying \eqref{LB}, \eqref{doublingraph} and \eqref{UEm} for some $m>2$. 
Then, for any $q\in (1,2)$,  there exists $c>0$ such that for all $y\in \Gamma$ and all $k\in \N^*$,
$$ \left\| \nabla_{\Gamma} p_{k-1}(.,y)\, e^{c\left(\frac{d^m(.,y)}{k} \right)^\frac1{m-1}} \right\|_q \lesssim \frac{m(y)}{\sqrt{k} V(y,k^{1/m})^{1-\frac1q}}.  $$
\end{lemma}

\begin{proof}
Let $q\in (1,2)$  and $y\in \Gamma$. In the sequel, we write $u_k(x)$ for $p_{k-1}(x,y)$.  We may and do assume that $k\geq 2$. 

\medskip

\noindent \underline{Step 1:} Estimate of the gradient by a  pseudo-gradient.

\noindent Let us define the  pseudo-gradient 
$$N_q u_k(x) = -u_k^{2-q}(x) [\partial_k +\Delta_{\Gamma}] u_k^q(x)$$
where $\partial_k$ denotes the discrete time differentiation defined by $\partial_k u_k = u_{k+1} - u_k$.
We also need the ``averaging'' operator 
$$A g(x) = \sum_{z \sim x} g(z).$$
Propositions 4.6 and 4.7 in \cite{Fen15} yield that for all $k\in \N^*$ and all $x\in \Gamma$, 
$$\left|\nabla_{\Gamma} u_k(x)\right|^2 \lesssim A [N_q u_k](x).$$
Therefore,
\[\begin{split}
\left\| \nabla_{\Gamma} u_k \,e^{c\left(\frac{d^m(.,y)}{k} \right)^\frac1{m-1}} \right\|_q  
& \lesssim  \left\| A [N_q u_k]\, e^{2c \left(\frac{d^m(.,y)}{k} \right)^\frac1{m-1}}\right\|_{q/2}^{1/2} \\
& \lesssim   \left\| A \left[N_q u_k \,e^{2c \left(\frac{[d(.,y)+1]^m}{k} \right)^\frac1{m-1}} \right]\right\|_{q/2}^{1/2}\\
& \lesssim  \left\| N_q u_k \,e^{2c\left(\frac{[d(.,y)+1]^m}{k} \right)^\frac1{m-1}} \right\|_{q/2}^{1/2},
\end{split} \]
where the last line holds because of the boundedness of the operator $A$ on $L^r$ for $r\in (0,1]$ (see for example \cite[Proposition 3.1]{Du08}).
Remark now that $[d+1]^{m/(m-1)} \leq 2^{1/(m-1)} [d^{m/(m-1)} + 1]$. Thus, if $c' = 2^{\frac{m}{m-1}} c$, we have
\begin{equation} \label{Step1} \left\| \nabla_{\Gamma} u_k e^{c\left(\frac{d^m(.,y)}{k} \right)^\frac1{m-1}} \right\|_q 
\lesssim \left\| N_q u_k e^{c' \left(\frac{d(.,y)^m}{k} \right)^\frac1{m-1}} \right\|_{q/2}^{1/2}. 
 \end{equation}

\medskip

\noindent \underline{Step 2:} A Hardy-Stein type identity. 

\noindent Let  $J_k = -[\partial_k + \Delta_{\Gamma}] u_k^q$.  Proposition 4.7 in \cite{Fen15} yields that $N_q u_k$  is non-negative, thus so is $J_k$.

\noindent Since  $u_k \in L^q(\Gamma)$,  that is  $u_k^q \in L^1(\Gamma)$,  we have  $\sum_{x\in \Gamma} \Delta_{\Gamma} u_k^q(x) m(x) = 0$.  
Therefore, with H\"older inequality,
\[\begin{split}
\sum_{x\in \Gamma} J_k(x) m(x) & = -\sum_{x\in \Gamma} \partial_k(u_k^q)(x) m(x) \\
& \leq - q \sum_{x\in \Gamma} u_k^{q-1}(x) \partial_k u_k(x) m(x)  \\
& \lesssim \|u_k\|_q^{q-1} \|\partial_k u_k\|_q = \|u_k\|_q^{q-1} \|\Delta_{\Gamma} u_k\|_q, 
\end{split}\]
where the second line is a consequence of  Young's inequality. 
As a consequence, \eqref{Step1} becomes
\begin{equation} \label{Step2}\begin{split}
 \left\| \nabla_{\Gamma} u_k e^{c\left(\frac{d^m(.,y)}{k} \right)^\frac1{m-1}} \right\|_q^q 
& \lesssim  \left\| N_q u_k e^{c' \left(\frac{d(.,y)^m}{k} \right)^\frac1{m-1}} \right\|_{q/2}^{q/2} \\
& \quad = \sum_{x\in \Gamma}  u_k(x)^{q(1-\frac{q}{2})} e^{\frac{c'q}{2} \left(\frac{d(x,y)^m}{k} \right)^\frac1{m-1}} J_k(x)^{q/2} m(x) \\
& \quad \leq  \left(\sum_{x\in \Gamma} u_k(x)^{q} e^{\frac{c'q}{2-q} \left(\frac{d(x,y)^m}{k} \right)^\frac1{m-1}} m(x) \right)^{1-\frac{q}2} \left(\sum_{x\in \Gamma} J_k(x) m(x) \right)^{q/2} \\
& \lesssim  \left(\sum_{x\in \Gamma} u_k(x)^{q} e^{\frac{c'q}{2-q} \left(\frac{d(x,y)^m}{k} \right)^\frac1{m-1}} m(x) \right)^{1-\frac{q}2}  \|u_k\|_q^{q(q-1)/2} \|\Delta_{\Gamma} u_k\|_q^{q/2}. 
\end{split}\end{equation}

\medskip

\noindent \underline{Step 3:} Conclusion.

\noindent The use of \eqref{UEm}, as well as the discrete version of \eqref{est2}, yields, if $c$ (and thus $c'$) is small enough,
\[\begin{split}
\sum_{x\in \Gamma}  u_k(x)^{q} e^{\frac{c'q}{2-q} \left(\frac{d(x,y)^m}{k} \right)^\frac1{m-1}} m(x) 
& \lesssim \frac{m^q(y)}{V(y,k^{1/m})^q}   \sum_{x\in \Gamma} e^{-c" \left(\frac{d(x,y)^m}{k} \right)^\frac1{m-1}} m(x) \\
& \lesssim  \frac{m^q(y)}{V(y,k^{1/m})^{q-1}}.
\end{split} \]
As a consequence, 
$$ \|u_k\|_q^q \lesssim \frac{m^q(y)}{V(y,k^{1/m})^{q-1}}.$$
Finally, let $l \simeq n \simeq \frac{k}{2}$  (recall that $k\geq 2$).  
Since \eqref{LB} implies the  $L^q$ analyticity of $\Delta_{\Gamma}$ (see \cite[p. 426]{CSC90}),  one has
$$\|\Delta_{\Gamma} u_k\|_q^q \lesssim \frac{1}{l^q} \| u_{n}\|_q^q \lesssim \frac{m^q(y)}{l^q V(y,n^{1/m})^{q-1}} \lesssim \frac{m^q(y)}{k^q V(y,k^{1/m})^{q-1}},$$
where the last inequality holds thank to the doubling property \eqref{doublingraph}.

Using the last three estimates in \eqref{Step2}, one has
$$\left\| \nabla_{\Gamma} u_k e^{c\left(\frac{d^m(.,y)}{k} \right)^\frac1{m-1}} \right\|_q^q  \lesssim  \frac{m^q(y)}{k^{q/2} V(y,k^{1/m})^{q-1}}, $$
which concludes the proof.
\end{proof}


\section{The reverse inequality} \label{Vicsek}

A consequence of Theorems \ref{main} and \ref{maingraph}  through the duality argument recalled in the introduction  is that \eqref{RRp} holds for all $p\in [2,+\infty)$ under the assumptions of these theorems.  In this section, we exhibit a class of manifolds (respectively of graphs) satisfying the assumptions of Theorem \ref{main} (respectively of Theorem \ref{maingraph}) on which \eqref{RRp} is false whenever $1<p<2$ (therefore, again by duality, \eqref{Rp}  is false for $2<p<+\infty$).

 Our model space will be the infinite Vicsek graph built in $\R^N$ whose first steps of construction are shown in Figure \ref{fig1}.
A Vicsek graph has polynomial volume growth of exponent  $D=\log_3 (1+2^N)$.
\begin{figure}
\begin{center}
\includegraphics{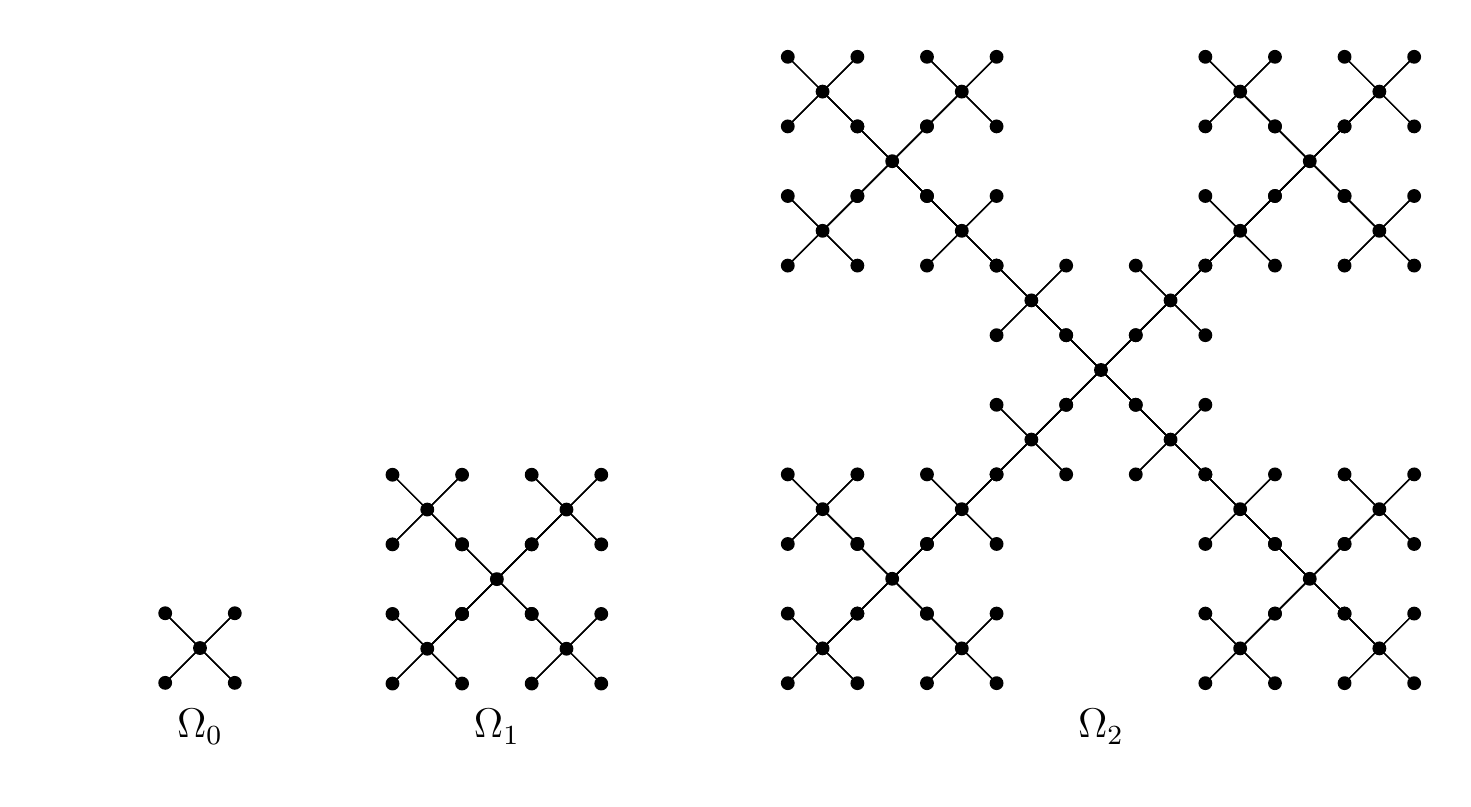}
\end{center}
\caption{The first steps of the Vicsek graph with parameter $D=\log_3 5$  built in $\R^2$.}\label{fig1}
\end{figure}

 As in \cite{BCG01}, we will consider the associated Vicsek manifold built by replacing edges by tubes.  

Such a  manifold $M$ has polynomial volume growth of exponent $D$ and satisfies the following heat kernel estimate (see \cite[Proposition 5.2]{BCK05}):
\[
h_t(x,y) \lesssim \frac{1}{t^{\frac{D}{D+1}}} \exp\br{-c\br{\frac{d^{D+1}(x,y)}{t}}^{1/D}},\,t \geq 1,
\]
in other words, it satisfies \eqref{B} with $m=D+1$.

Also, it satisfies the following non-standard Poincar\'e inequality 
\begin{equation}\label{reP}
\int_B |f-f_B|^2 d\mu \lesssim  r_B^{D+1} \int_B |\nabla f|^2 d\mu,\,\,\forall r_B\geq 1,\,\,\forall f \in \mathcal 
C_0^{\infty}(M),
\end{equation}
where $f_E$ denotes the average of $f$ on the set $E$ (\cite[Section 5]{CSC93}, \cite[Theorem 1.2]{BB04}). Note that \eqref{reP} is weaker than the $L^2$ Poincar\'e inequality:
\begin{equation}
\label{P}\tag{$P$}
\int_B |f-f_B|^2 d\mu \lesssim r_B^2 \int_B |\nabla f|^2 d\mu, \,\,\forall f \in \mathcal 
C_0^{\infty}(M),
\end{equation}
but it is optimal on $M$ (see  for example \cite[Remark 5.2]{ChThese}).

On Vicsek manifolds, Theorem \ref{main} tells us that (\ref{Rp})  holds  for $1< p\leq 2$. The following is a negative result for $p>2$.  It comes from  \cite[Th\'eor\`eme 0.30]{ChThese}. 
\begin{theorem}\label{neg}
Consider a Vicsek manifold $M$  with the polynomial volume growth $V(x,r) \simeq r^D$ for $r\geq 1$. Let $\beta\in (\frac{1}{D+1},\frac{D}{D+1})$. Then the inequality
\begin{equation}
\label{eRRp}
\Vert \Delta^{\beta} f\Vert_p \lesssim \Vert |\nabla f |\Vert_p,\,\,\forall f \in \mathcal 
C_0^{\infty}(M),
\end{equation}
is false for $1<p<\frac{D-1}{\beta(D+1)-1}$. 

In particular, choosing $\beta=1/2$, for all $p\in (1,2)$, \eqref{RRp} does not hold.
Consequently, the Riesz transform is not bounded on $L^p$ for any $2<p<+\infty$. 
\end{theorem}

A similar statement holds for  Vicsek graphs.

\begin{remark}
\begin{enumerate}

\item  The first part of the following proof is taken from \cite[Proposition 6.2]{CD03}, where the authors proved that (\ref{RRp}) is false for $p\in (1, \frac{2D}{D+1})$. We will rewrite it for the sake of completeness. Note that the conclusion of \cite[Proposition 6.2]{CD03} is weaker that the one of Theorem \ref{neg}. 

\item Recall that Auscher and the second author proved in \cite{AC05}  that under  (\ref{doubling}) and (\ref{P})  the Riesz transform is bounded on $L^p$ for $2<p<2+\varepsilon$, where $\varepsilon>0$. But of course the Vicsek manifold does not satisfy \eqref{P} (as we already said, the non-standard Poincar\'e inequality \eqref{reP} is optimal on $M$). By contrast, Theorem \ref{neg} shows that the conjunction of \eqref{doubling} and \eqref{reP} does not imply the existence of $\varepsilon>0$ such that \eqref{Rp} holds for $p\in (2,2+\varepsilon)$.
\end{enumerate}
\end{remark}

\begin{proof}  Consider a Vicsek manifold $M$ modelled on a Vicsek graph $\Gamma$ built in $\R^N$ with polynomial volume growth of exponent $D$.
Recall (\cite[Section 6]{BCG01}) that $M$ satisfies
\[
h_t(x,x) \lesssim t^{-\frac{D}{D+1}},\,t \geq 1, \ x\in M.
\]
Let $D'=\frac{2D}{D+1}$. Therefore, we have for $p>1$
\[
\Vert e^{-t\Delta} \Vert_{1 \rightarrow p} \lesssim t^{-\frac{D'}{2}\left( 1-\frac{1}{p}\right)}, \,t \geq 1.
\]
This implies the following Nash inequality (see \cite{Co92}): for all $\beta>0$, it holds
\[
\Vert f\Vert_p^{1+\frac{2 \beta p}{(p-1)D'}} 
\lesssim  \Vert f \Vert_1^{2 \beta \frac{p}{(p-1)D'}} \Vert \Delta^{\beta} f \Vert_p,\,
\forall f \in \mathcal{C}_0^{\infty}(M) \text{ such that } \frac{\Vert f \Vert_p}{\Vert f \Vert_1} \leq 1.
\]

Let $\beta\in (\frac{1}{D+1},\frac{D}{D+1})$ and assume that (\ref{eRRp}) holds. It follows 
\begin{equation}
\label{nash}
\Vert f\Vert_p^{1+\frac{2\beta p}{(p-1)D'}} 
\lesssim  \Vert f \Vert_1^{\frac{2\beta p}{(p-1)D'}} \Vert |\nabla f |\Vert_p,\,
\forall f \in \mathcal{C}_0^{\infty}(M) \text{ such that } \frac{\Vert f \Vert_p}{\Vert f \Vert_1} \leq 1.
\end{equation}

The Vicsek graph $\Gamma$ is endowed with  the standard weight $m$ (that is, $\mu_{xy}=1$ if $x,y$ are neighbours, $0$ otherwise,
hence $m(x)$ is nothing but the number of neighbours of $x$) out of which $M$ is constructed by replacing the edges with tubes. For any vertex set $\Omega \subset \Gamma$, $m(\Omega)\simeq |\Omega|$, where $|\Omega|$ is the cardinality of $\Omega$. We learn from \cite[Section 6]{CSC95}  that (\ref{nash}) implies the following analogue of \eqref{nash} on $\Gamma$: 

\begin{equation}
\label{gnash}
\Vert g\Vert_p^{1+\frac{2\beta p}{(p-1)D'}} 
\lesssim  \Vert g \Vert_1^{\frac{2\beta p}{(p-1)D'}} \Vert \nabla_{\Gamma} g \Vert_p, \,  \forall \,g\in c_0(\Gamma).
\end{equation}

We will show that there exists a family of functions that disproves (\ref{gnash}) hence  \eqref{eRRp}.

Indeed, take the same $\Omega_n$ and $g_n$ as in \cite[Section 4]{BCG01} (see Figure \ref{fig2}). That is: 
$\Omega_n=\Gamma \bigcap [0,3^n]^N$, where $2^N+1=3^D$. Then $|\Omega_n| \simeq 3^{D n}$. Denote by $z_0$ 
the center of $\Omega_n$ and by $z_i , 1\leq i \leq 2^N$ its corners. Note that $d(z_0,z_i)=3^n$. Define  $g_n$ as follows: $g_n(z_0) = 1, g_n(z_i) = 0,  1\leq i \leq 2^N$, and extend  $g_n$ 
as a harmonic function in the rest of $\Omega_n$.  More precisely, if $z$ belongs to the diagonal linking $z_0$ and $z_i$, then $g_n(z)=3^{-n}d(z_i,z)$.  Otherwise, $g_n(z)=g_n(z')$  where $z'$ is the only vertex satisfying $d(z,z') = \inf\{d(z,y), \ y \text{ belongs to a diagonal of $\Omega_n$}\}$.

\begin{figure}
\begin{center}
\includegraphics{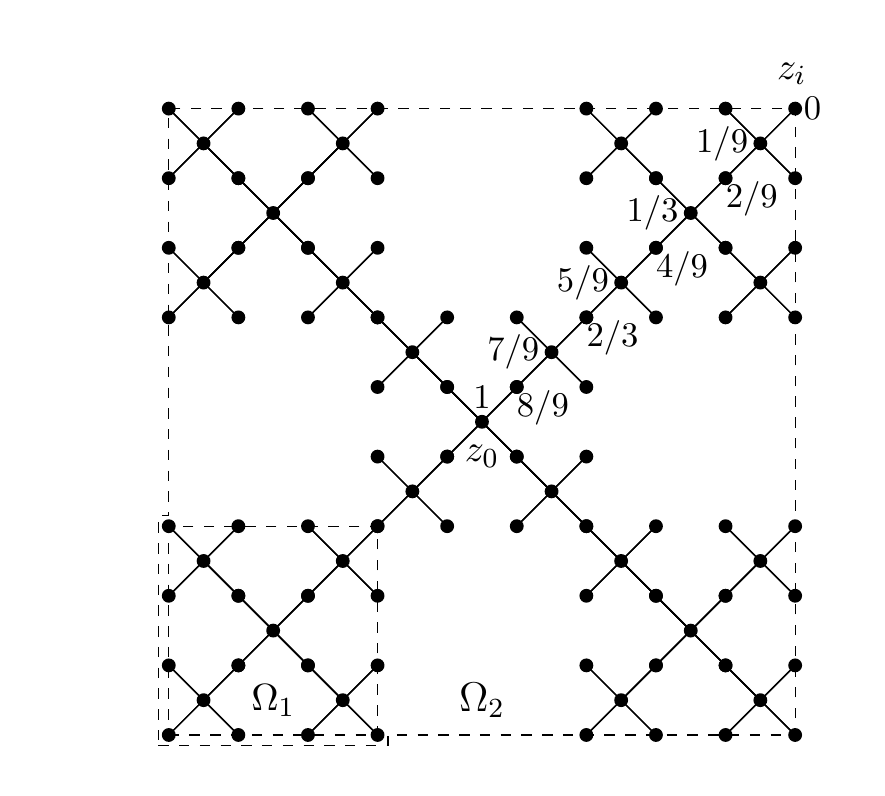}
\end{center}
\caption{the function $g_2$ on the diagonal $z_0 z_i$} \label{fig2}
\end{figure}

On the one hand, we have 
\[
\sum_{x\in \Gamma} |g_n(x)| m(x)\leq m(\Omega_{n}) \simeq |\Omega_{n}|. 
\] 
On the other hand, for any $x$ in the $n-1$  block with centre $z_0$, we have   $g_n(x) \geq \frac{2}{3}$. 

\noindent Therefore
\[
\sum_{x\in \Gamma} |g_n(x)|^p m(x)  \geq (2/3)^p m(\Omega_{n-1}) \simeq 3^{D n} \simeq |\Omega_{n}|.
\]

Since, whenever $x\sim y$,  $|g_n(x)-g_n(y)|  = 3^{-n}$ if $x, y$ belong to the same diagonal connecting $z_0$ and $z_i$,
and otherwise  $g_n(x)-g_n(y)= 0$, we obtain
\[
\Vert \nabla_\Gamma g_n\Vert_{p}^p \simeq \sum_{x\sim y}|g_n(x)-g_n(y)|^p
= \sum_{i=1}^{2^N} 3^{-np} d(z_0,z_i)=2^N 3^{-n(p-1)} \simeq |\Omega_n|^{-\frac{p-1}{D}}.
\]
Thus (\ref{gnash}) yields
\begin{eqnarray}
\label{exp}
|\Omega_n|^{\frac{1}{p}\left(1+\frac{2\beta p}{(p-1)D'}\right)}
\lesssim |\Omega_n|^{\frac{2\beta p}{(p-1)D'}-\frac{1}{p'D}}.
\end{eqnarray}

Obviously $|\Omega_n|$ tends to infinity as  $n\rightarrow \infty$, hence \eqref{exp} cannot hold if the RHS exponent is smaller than the LHS one, that is if
\begin{eqnarray*}
\frac{2\beta p}{(p-1)D'}-\frac{1}{p'D} - \frac{1}{p}\left(1+ \frac{2\beta p}{(p-1)D'}\right)
& = & \left(1-\frac{1}{p}\right)\frac{2\beta p}{(p-1)D'} - \frac{1}{p} - \frac{1}{p'D}
\\ & = & \beta\frac{D+1}{D} - \frac{1}{p} - \frac{1}{p'D}
<0.
\end{eqnarray*}

In other words, (\ref{exp}) is false for $1<p<\frac{D-1}{\beta(D+1)-1}$, where 
$\beta\in (\frac{1}{D+1},\frac{D}{D+1})$. This contradicts our assumption (\ref{RRp}) on $\Gamma$ and $M$.

In particular,  taking $\beta=1/2$ shows that   (\ref{exp}) is false for $1<p<2$. Thus (\ref{RRp}) is false for  $1<p<2$. This also 
indicates that the Riesz transform is not bounded on $L^{p}$ for $2<p< +\infty $. As we already said, the statement on (\ref{Rp}) follows  by duality.
\end{proof}

 Combining Theorem \ref{main} and Theorem \ref{neg}, we have a full picture for the comparison of $\norm{|\nabla f|}_p$ and $\norm{\Delta^{1/2} f}_p$ on Vicsek manifolds and graphs. That is
\begin{theorem}
For any Vicsek manifold, and  $p\in(1,+\infty)$,  \eqref{Rp} holds if and only if $p\le 2$ and \eqref{RRp} holds if and only if $p\ge 2$.
\end{theorem}

A similar statement holds for  Vicsek graphs.

\appendix

\section*{Appendix: Estimates for the heat kernel in fractal manifolds}

Let $(\Gamma,\mu)$ be a graph as in Section \ref{Rieszgraph}. For all $A\subset \Gamma$, let $\partial A$ denote the exterior boundary of $A$, defined as $\left\{x\in \Gamma\setminus A;\mbox{ there exists }y\sim x\mbox{ with }y\in A\right\}$, and let $\overline{A}:=A\cup \partial A$. \par
\noindent A function $h:\overline{A}\rightarrow \R$ is said to be harmonic in $A$ if and only if, for all $x\in A$, $\Delta h(x)=0$. \par
\noindent Say that $(\Gamma,\mu)$ satisfies a Harnack elliptic inequality if and only if, for all $x\in \Gamma$, all 
$R\geq 1$ and all nonnegative harmonic functions $u$ in $B(x,2R)$,  
\begin{equation} \label{EHI} \tag{$EHI$}
\sup_{B(x,R)} u\lesssim \inf_{B(x,R)} u.
\end{equation}
Let $D>0$. Say that $\Gamma$ satisfies $(V_{D})$ if and only if, for all $x\in \Gamma$ and all $r>0$,
\begin{equation} \label{Ahlfors} \tag{$V_{D}$}
V(x,r)\sim r^{D}.
\end{equation}
Denote by $(X_n)_{n\geq 1}$ a random walk on $\Gamma$, that is a Markov chain with transition probability given by $p$. For all $A\subset \Gamma$, define
$$
T_A:=\min\left\{n\geq 1;\ X_n\in A\right\}
$$
and, for all $x\in \Gamma$ and all $r>0$, let
$$
\tau_{x,r}:=T_{B^c(x,r)}.
$$
If $m>0$, say that $\Gamma$ satisfies \eqref{exit} if and only if, for all $x\in \Gamma$ and all $r>0$,
\begin{equation} \label{exit} \tag{$E_{m}$}
\E^x \tau_{x,r}\sim r^{m}.
\end{equation}
Theorem 2 in \cite{Ba04} claims that, for all $m\in [2,D+1]$, there exists a graph $\Gamma$ satisfying \eqref{Ahlfors} (hence the doubling volume property), \eqref{EHI} and \eqref{exit}. Therefore, Theorem 2.15 in \cite{BBK06} (see also \cite[Theorem 3.1]{gt}) implies that $\Gamma$ satisfies the following parabolic Harnack inequality: for all $x_0\in \Gamma$, all $R\geq 1$ and all non-negative functions $u:\llbracket 0,4N\rrbracket\times \overline{B(x_0,2R)}$ solving $u_{n+1}-u_n=\Delta u_n$ in $\llbracket 0,4N-1\rrbracket\times B(x_0,2R)$, one has
$$
\max_{n\in \llbracket N,2N-1\rrbracket,\ y\in B(x,R)} u_n(y)\lesssim \min _{n\in \llbracket 3N,4N-1\rrbracket,\ y\in B(x,R)} (u_n(y)+u_{n+1}(y)),
$$
where $N$ is an integer satisfying $N\sim R^{m}$ and $N\geq 2R$. \par
\noindent Consider now the manifold $M$ built from $\Gamma$ with a self-similar structure at infinity by replacing the edges of the graph with tubes of length $1$ and then gluing the tubes together smoothly
at the vertices. Since $M$ and $\Gamma$ are roughly isometric, Theorem 2.21 in \cite{BBK06} yields that $M$ satisfies the following parabolic Harnack inequality: for all $x_0\in M$ and all $R>0$, for all non-negative solutions $u$ of $\partial_tu=\Delta u$ in $(0,4R^{m})\times B(x_0,R)$, one has
$$
\sup_{Q_-}u\lesssim \inf_{Q_+}u,
$$
where
$$
Q_-=(R^{m},2R^{m})\times B(x_0,R)
$$
and
$$
Q_+=(3R^{m},4R^{m})\times B(x_0,R).
$$
In turn the parabolic Harnack inequality implies \eqref{B} (see \cite[Theorem 5.3]{HSC01}). See also \cite[Section 1.2.7]{Post}.

Finally, let us explain why   \eqref{B}  for $m>2$ is incompatible with \eqref{UE}. It is classical (see for instance \cite[Section 3]{coupre}) that \eqref{B} together with \eqref{doubling} implies the on-diagonal lower bound
$$h_{t}(x,x) \gtrsim  \frac{1}{V(x,t^{1/m})}, \ t\geq 1,$$
which is clearly incompatible with 
$$h_{t}(x,x) \lesssim  \frac{1}{V(x,t^{1/2})}, \ t\geq 1$$
because of the so-called reverse volume doubling property (see for instance \cite[Proposition 5.2]{GH14}).


\begin{thebibliography}{10}
\bibitem{Aus07} P. Auscher, \em On necessary and sufficient conditions for $L^p$-estimates of Riesz transforms associated to elliptic operators on $\mathbb R^n$ and related estimates, \em Mem. Amer. Math. Soc., {\bf 186} (2007), no. 871.

\bibitem{AC05}
P. Auscher and T. Coulhon,  Riesz transform on manifolds and {P}oincar{\'e} inequalities, 
{\em Ann. Sc. Norm. Super. Pisa Cl. Sci.}, (5), \textbf{4(3)} (2005), 531--555.




\bibitem{Ba04}
M. Barlow,  Which values of the volume growth and escape time exponent are possible for a graph?, {\em Rev. Mat. Iberoamericana},
\textbf{20(1)} (2004), 1--31.

\bibitem{BB04}
M. Barlow and R. Bass,
Stability of parabolic Harnack inequalities, {\em Trans. Amer. Math. Soc.},
\textbf{356(4)} (2004), 1501--1533.

\bibitem{BBK06} M. Barlow, R. Bass, and T. Kumagai, Stability of parabolic Harnack inequalities on metric measure spaces, {\em J. Math. Soc. Japan}, \textbf{58(2)} (2006), 485--519. 

\bibitem{BCG01} M. Barlow, T. Coulhon, and A. Grigor'yan, Manifolds and graphs with slow heat kernel decay, {\em  Invent. Math.}, \textbf{144} (2001), no. 3, 609--649.

\bibitem{BCK05} M. Barlow, T. Coulhon, and T. Kumagai, Characterization of sub-Gaussian heat kernel estimates on strongly recurrent graphs, {\em  Comm. Pure App. Math.}, \textbf{58} (2005), no. 12, 1642--1677. 

\bibitem{BK03} S. Blunck and P. C. Kunstmann,  Calder\'on-Zygmund theory for non-integral operators and the $H^\infty$ functional calculus, {\em Rev. Mat. Iberoamericana}, {\bf 19} (2003), no. 3, 919--942.

\bibitem{Ch15}
L. Chen,  Sub-Gaussian heat kernel estimates and quasi Riesz transforms for $1\leq p\leq 2$, 
  {\em Publ. Mat.}, \textbf{59} (2015), 313-338.

\bibitem{ChThese}
L. Chen, {\em Quasi Riesz transforms, Hardy spaces and generalized sub-Gaussian heat kernel estimates}, PhD thesis,
Universit\'e Paris Sud - Paris XI; Australian national university, 2014. https://tel.archives-ouvertes.fr/tel-01001868.

\bibitem{CW71}
R. Coifman and G. Weiss, {\em Analyse harmonique non-commutative sur certains espaces homog\`{e}nes}, 
Lecture notes in Math., \textbf{242}, Springer-Verlag, Berlin, 1971.

\bibitem{Co92} T. Coulhon,  In\'egalit\'es de Gagliardo-Nirenberg pour les semi-groupes d'op\'erateurs et 
applications, {\em Potential Anal.}, \textbf{1} (1992), no. 4, 343--353.

\bibitem{coupre}
T.~Coulhon,
\newblock Off-diagonal heat kernel lower bounds without Poincar\'e,
\newblock{\em J. London Math. Soc.}, \textbf{68} (2003), no. 3, 795--816.


\bibitem{CD99}
T. Coulhon and X. T. Duong, Riesz transforms for {$1\leq p\leq 2$}, 
{\em Trans. Amer. Math. Soc.}, \textbf{351(3)} (1999), 1151--1169.

\bibitem{CD03} T. Coulhon and X. T. Duong,  Riesz transform and related inequalities on noncompact 
Riemannian manifolds, {\em Comm. Pure Appl. Math.}, \textbf{56} (2003), no. 12, 1728--1751.

\bibitem{CDL}  T. Coulhon, X. T. Duong, and X.-D. Li,  Littlewood-Paley-Stein functions on complete Riemannian manifolds  for $1\le p\le 2$,
 {\it Studia Math.}, {\textbf 154 (1)} (2003), 37--57. 


\bibitem{CSC90} T. Coulhon and L. Saloff-Coste, Puissances d'un op\'erateur 
r\'egularisant,  {\it Ann. Inst. H. Poincar\'e Probab. Statist.}, \textbf{26 (3)} (1990), 419--436.  

\bibitem{CSC93} T. Coulhon and L. Saloff-Coste,  Isop\'erim\'etrie pour les groupes et les vari\'et\'es, 
{\em Rev. Mat. Iberoamericana}, \textbf{9(2)} (1993), 293--314.

\bibitem{CSC95} T. Coulhon and L. Saloff-Coste, Vari\'et\'es riemanniennes isom\'etriques \`a l'infini, 
{\em Rev. Mat. Iberoamericana}, \textbf{11} (1995), no. 3, 687--726.

\bibitem{Du08}
N. Dungey,  A Littlewood-Paley-Stein estimate on graphs and groups, {\em Studia Math.}, \textbf{189(2)} (2008), 113--129.

\bibitem{DM99} 
X. T. Duong and A. McIntosh,  Singular integral operators with non-smooth kernels on irregular domains,
{\em Rev. Mat. Iberoamericana}, \textbf{15(2)} (1999), 233--263.

\bibitem{DR} X. T. Duong and D. Robinson, Semigroup kernels, Poisson bounds, and holomorphic functional calculus, {\em J. Funct. Anal.}, \textbf{142} (1996), 89--128.


\bibitem{Fen15} J. Feneuil,  Littlewood-Paley functionals on graphs, {\em Math. Nachr.}, \textbf{288(11-12)} (2015), 1254--1285. 

\bibitem{F15}
J. Feneuil,  Riesz transform on graphs under subgaussian estimates, arXiv:1505.07001.

\bibitem{FeSpectre} 
J. Feneuil,  On diagonal lower bound of Markov kernel from $L^2$ analyticity, arXiv:1508.02447.


\bibitem{Gr95}
A. Grigor'yan,  Upper bounds of derivatives of the heat kernel on an arbitrary complete
manifold, {\em J. Funct. Anal.}, \textbf{127} (1995), 363--389.

\bibitem{Gr09} 
A. Grigor'yan, {\em Heat kernel and analysis on manifolds}, AMS/IP Studies in Advanced Mathematics \textbf{47}, Amer. Math.  Soc., Providence, RI, 2009.

\bibitem{GH14} 
A. Grigor'yan and J. Hu,  Upper bounds of heat kernels on doubling spaces, {\em Moscow Math. J.}, \textbf{14} (2014), 505--563.


\bibitem{gt} A. Grigor'yan, A. Telcs, Harnack inequalities and sub-Gaussian estimates for random walks, {\em Math. Ann.}, \textbf{324} (2002), 521--556. 


\bibitem{HSC01} 
W. Hebisch and L. Saloff-Coste,  On the relation between elliptic and parabolic Harnack inequalities, 
{\em Ann. Inst. Fourier}, \textbf{51(5)} (2001), 1437--1481.

\bibitem{Post}  O. Post, {\em Spectral analysis on graph-like spaces}, Lecture notes in Math., \textbf{2012}, Springer-Verlag, Berlin, 2012. 

\bibitem{Russ00} E. Russ, Riesz transforms on graphs for $1\leq p\leq 2$, {\em Math. Scand.}, \textbf{87(1)} (2000), 133--160.


\bibitem{Sjo99} P. Sj\"ogren,  An estimate for a first-order Riesz operator on the affine group, {\em Trans. Amer. Math. Soc.}, {\bf 351} no. 8 (1999), 3301--3314.

\bibitem{SV08} P. Sj\"ogren and M. Vallarino,  Boundedness from $H^1$ to $L^1$ of Riesz transforms on a Lie group of exponential growth, {\em Ann. Inst. Fourier}, {\bf 58} no. 8 (2008), 1117--1151.

\bibitem{St701} 
E. M. Stein,
{\em Singular integrals and differentiability properties of functions}, 
Princeton Mathematical Series, \textbf{30}, Princeton Univ. Press, Princeton, N. J., 1970.

\bibitem{St70} 
E. M. Stein, {\em Topics in harmonic analysis related to the Littlewood-Paley theory}, 
Annals of Mathematics Studies, \textbf{63}, Princeton Univ. Press, Princeton, N. J., 1970.

\bibitem{Str83} 
R. S. Strichartz, Analysis of the Laplacian on the complete Riemannian manifold, {\em J. Funct. Anal.}, \textbf{52(1)} (1983), 48--79.

\end{thebibliography}
\end{document}